\documentclass{birkjour}
 \newtheorem{thm}{Theorem}[section]
 
 \newtheorem{lem}[thm]{Lemma}
 \newtheorem{prop}[thm]{Proposition}
 \theoremstyle{definition}
 \newtheorem{defn}[thm]{Definition}
 \theoremstyle{remark}
 \newtheorem{rem}[thm]{Remark}
 
 \newtheorem{open}[thm]{Open problem}
 \numberwithin{equation}{section}
\newcommand{\R}{\mathbb R}
\newcommand{\N}{\mathbb{N}}
\def\ds{\displaystyle}
\newcommand{\cH}{{\mathcal H}}
\def\eq#1{(\ref{#1})}
\newcommand{\innt}{\text{\rm int}}
\newcommand{\eps}{\varepsilon}

\begin{document}

%
%
%
%
%
%
%
%
%

\title[Nonlinear elliptic equations on Riemannian models]
 {Partial symmetry and existence\\ of least energy solutions to \\ some nonlinear elliptic equations \\ on Riemannian models}

\author{Elvise Berchio}

\address{%
Dipartimento di Scienze Matematiche \\ Politecnico di Torino\\
Corso Duca degli Abruzzi 24\\
 10129 Torino\\
  Italy}

\email{elvise.berchio@polito.it}

\author{Alberto Ferrero}
\address{Dipartimento di Scienze e Innovazione Tecnologica\\
Universit\`a del Piemonte Orientale ``Amedeo Avogadro''\\
 Viale Teresa Michel 11\\
 15121 Alessandria\\
  Italy}
\email{alberto.ferrero@mfn.unipmn.it}

\author{Maria Vallarino}

\address{%
Dipartimento di Scienze Matematiche\\
 Politecnico di Torino\\
Corso Duca degli Abruzzi 24\\
 10129 Torino\\
  Italy}

\email{maria.vallarino@polito.it}


\date{January 1, 2004}

\begin{abstract}
We consider least energy solutions to the nonlinear equation $-\Delta_g u=f(r,u)$ posed on a class of Riemannian models $(M,g)$ of dimension $n\ge 2$ which include the classical hyperbolic space $\mathbb H^n$ as well as manifolds with unbounded sectional geometry. Partial symmetry and existence of least energy solutions is proved for quite general nonlinearities $f(r,u)$, where $r$ denotes the geodesic distance from the pole of $M$.
\end{abstract}

\maketitle

\section{Introduction}
Let $(M,g)$ be a $n$-dimensional Riemannian model ($n\geq 2$),
namely a manifold admitting a pole $o$ and whose metric is given,
in spherical coordinates around $o$, by
\begin{equation}\label{metric}
ds^2=dr^2+(\psi(r))^2d\Theta^2,\ \ r>0, \Theta\in {\mathbb S}^{n-1}\,,
\end{equation}
where $d\Theta^2$ denotes the canonical metric on the unit
sphere ${\mathbb S}^{n-1}$ and
\begin{itemize}
\item[$(H)$] $\psi$ is a $C^{\infty}$ nonnegative function on $[0,\infty)$, positive on $(0,\infty)$ such that
$\psi'(0)=1$ and $\psi^{(2k)}(0)=0$ for all $k\geq 0$.
\end{itemize}
These conditions on $\psi$ ensure that the manifold is smooth and
the metric at the pole $o$ is given by the euclidean metric
\cite[Chapter 1, 3.4]{P}. Then, by construction, $r:=d(x,o)$ is
the geodesic distance between a point $x$ whose coordinates are
$(r,\Theta)$ and $o$. \par
Let $\Delta_g$ denote the Laplace-Beltrami
operator on $M$. Our paper concerns least energy solutions to the equation
\begin{equation}\label{equa1}
-\Delta_g u=f(r,u)\ \ \ {\rm on}\ M\,.
\end{equation}
As a prototype of the nonlinearity think to
$f(r,u)=W(r)|u|^{p-1}u$, where $W$ is a suitable measurable
function and $1< p\leq \frac{n+2}{n-2}$ if $n\geq 3$ ($1<p$ if $n=2$), but most of the
results stated in the paper hold for more general $f$, see Section \ref{partial sym}. Nonlinear
elliptic equations like \eq{equa1} on manifolds with negative
sectional curvatures have been the subject of intensive research
in the past few years. Many papers are settled on the simplest
example of manifold with negative curvature: the hyperbolic space
$\mathbb{H}^n$, corresponding to $\psi(r)=\sinh r$ in \eq{metric}.
See \cite{BS,GS, GS2,mancini} and references therein, where
$f=f(u)=\lambda u+|u|^{p-1}u$ is chosen. In this case, a great
attention has been devoted to the study of radial solutions (non
necessarily in the energy class) either in $\mathbb{H}^n$
\cite{BK,BGGV,mancini} or in the more general Riemannian model \eq{metric} \cite{bfg}. See also \cite{Punzo} where fully nonlinear elliptic equations have been recently studied in the same setting \eq{metric}. It becomes then a natural and interesting subject of investigation the study of symmetry properties of solutions to \eqref{equa1}.

In the hyperbolic setting, radial symmetry of solutions has been proved in \cite{adg,mancini} for power-type nonlinearities and for positive
solutions in the energy class. See also \cite{cfms}. The results
in \cite{adg} hold for quite general nonlinearities $f=f(u)$ and
non-energy solutions are also dealt. Furthermore, their extension
to general manifolds is also discussed. In the wake of the seminal
paper \cite{gidas}, the proofs of the just mentioned results rely
on the moving plane method and strongly exploit the structure of
the space under consideration. Hence, their extension to general
manifolds seems quite difficult to be reached. In \cite{adg} this topic is addressed by requiring two kinds of assumptions: either group action properties, which generalize what happens in $\mathbb{R}^n$ and $\mathbb{H}^n$, or suitable foliation conditions.\par
Coming back to our Riemannian model \eq{metric}, the results in \cite{adg} only apply if $\psi(r)=r$ or $\psi(r)=\alpha^{-1} \sinh(\alpha r)$ ($\alpha>0$), namely to the euclidean and hyperbolic cases, see Open Problem \ref{o:symmetry}. It is therefore appropriate to investigate whether, at least, some
partial symmetry holds. In the present paper, under quite general
assumptions on $\psi$ and $f$, we prove
that ground states to \eq{equa1} are foliated Schwarz symmetric
with respect to some point (see Theorem \ref{main}). In particular, they are either radial
symmetric or axially symmetric. The same can be said for
corresponding Dirichlet boundary value problems (see Theorem
\ref{main dir}). We refer to \cite{castorina2} for related results about
Dirichlet problems on Riemannian models. We observe that our symmetry result admits nonlinearities of the type $f=f(r,u)$ with no monotonicity condition with respect to $r$. As far as we are aware, this case was not covered by previous works, not even in the hyperbolic space $\mathbb H^n$. We mention the paper \cite{ago} where symmetry was proved for the solutions to a Dirichlet problem posed on manifolds conformally equivalent to $\R^n$ and for nonlinearities $f=f(r,u)$ decreasing with respect to $r$.\par
Our result guarantees that, when they exist, least energy solutions to
\eqref{equa1} are foliated Schwarz symmetric. The problem of
existence of least energy solutions to \eq{equa1} with subcritical growth can be easily
handled if radial symmetry is a-priori assumed (see \cite{bfg}).
In this perspective, for instance, compactness is gained in \cite{marzuola2} by requiring suitable symmetry
properties of solutions. If no extra constrain is assumed, the loss of compactness may represent a serious obstacle to show existence. When $f=f(u)=\lambda u+|u|^{p-1}u$ and $M=\mathbb{H}^n$, existence of
least energy solutions has been independently proved in
\cite{marzuola0} and in \cite{mancini}. Both the proofs exploit
peculiar properties of $\mathbb{H}^n$ and can be hardly extended
to a more general setting. An important contribution in this direction is given in \cite{marzuola}  where existence is proved for power-type nonlinearities when the equation is posed on
a weakly homogenous space. We show in Subsection \ref{ex} that, under the
weakly homogeneity assumption, our Riemannian model reduces either to $\R^n$ or $\mathbb{H}^n$. Nevertheless, a thorough analysis of the peculiar structure of \eq{metric} allows us to obtain some compactness and finally to prove in Theorem \ref{compact 1} existence of least energy solutions to \eq{equa1} for
suitable families of $f$ and for quite general $\psi$. It is worth
noticing that Theorem \ref{compact 1} applies to Riemannian models
with unbounded sectional geometry (see Remark
\ref{r:example}).

The paper is organized as follows. In Section \ref{setting} we fix the notation and describe our geometric setting. Section \ref{partial sym} contains the main theorems: in Subsection \ref{partial-symmetry} we state the partial symmetry results and in Subsection \ref{ex} we state the existence results. All the proofs are given in Sections 4-7.

\section{Notation and geometric setting}\label{setting}

\subsection{Notation}

The following table summarizes most of the notation we shall use in the paper.
\begin{itemize}

\item[-] For any $P\in M$ and $U_1,U_2\in T_P M$ we denote by $\langle U_1,U_2\rangle_g$ the scalar product on $T_P M$ associated with the metric $g$.

\item[-] For any $P\in M$ and $U\in T_P M$ we denote by $|U|_g:=\sqrt{\langle U,U\rangle_g}$ the norm of the vector $U$.

\item[-] $V_g$ denotes the volume measure in $(M,g)$.

\item[-] $\nabla_g$ denotes the Riemannian gradient in $(M,g)$.

\item[-] For any $1\le q <\infty$ we define the Banach space
$$
L^q(M):=\{u:M\to \R: u \text{ is measurable and } \int_M |u|^q dV_g<+\infty\}
$$
endowed with the corresponding $L^q$-norm.

\item[-] For any $1\le q<\infty$ and any measurable function $W:M\rightarrow [0,\infty)$ we define the Banach space
$$
L^q(M;W):=\{u:M\to \R: u \text{ is measurable and } \int_M |u|^q \, W\,dV_g<+\infty\}
$$
endowed with the corresponding weighted $L^q$-norm.

\item[-] We denote by $H^1(M)$ the classical Sobolev space in $M$, i.e.
$$
H^1(M)=\{u\in L^2(M): |\nabla_g u|_g \in L^2(M)\}
$$
endowed with the usual $H^1$-norm.

\end{itemize}

\subsection{Geometric setting}
Let $(M,g)$ be an $n$-dimensional Riemannian model whose metric is defined by formula \eqref{metric} with a function $\psi$ satisfying condition (H).
Let $\omega_n$ be the area of the $n$-dimensional unit sphere. Then
\[
S(r)=\omega_n (\psi(r))^{n-1},\ \ \ V(r)=\int_0^rS(t)\,{\rm d}t=\omega_n\int_0^r(\psi(t))^{n-1}\,{\rm d}t
\]
represent, respectively, the area of the geodesic sphere $\partial B_r$ and the volume of the geodesic ball $B_r$, where $B_r$ denotes the geodesic ball centered at $o$ of radius $r$, i.e.
$$
B_r:=\{(s,\Theta): 0\le s<r \text{ and } \Theta\in{\mathbb S}^{n-1}\}\,.
$$
The Riemannian Laplacian of a scalar function $f$ on $M$ is given, in the above coordinates, by
$$
\Delta_g
f(r,\Theta)=\frac1{(\psi(r))^{n-1}}\frac{\partial}{\partial
r}\left[(\psi(r))^{n-1} \frac{\partial f}{\partial
r}(r,\Theta)
\right]+\frac1{(\psi(r))^2}\Delta_{{\mathbb
S}^{n-1}}f(r,\Theta),
$$
where $\Delta_{{\mathbb S}^{n-1}}$ is the Riemannian Laplacian on the unit sphere ${\mathbb S}^{n-1}$.
In particular, for \it radial \rm functions, namely functions depending only on $r$, one has
\[
\Delta_g f(r)=\frac1{(\psi(r))^{n-1}}\left[(\psi(r))^{n-1}
f^\prime(r)\right]^\prime=f^{\prime\prime}(r)+(n-1)\frac{\psi^\prime(r)}{\psi(r)}f^\prime(r),
\]
where from now on a prime will denote, for radial functions,
derivative w.r.t. $r$. By standard arguments we deduce that the bottom of the $L^2$ spectrum of
$-\Delta_g$ in $M$ admits the following variational
characterization:
\begin{equation} \label{lambda1}
\lambda_1(M):= \inf_{\varphi \in C^{\infty}_{c}(M)\setminus
\{0\}}\frac{\int_{M} |\nabla_g \varphi |_{g}^2\,dV_{g}}{\int_{M}
\varphi^2\,dV_{g}}\,.
\end{equation}
Notice that if $\liminf_{r\rightarrow
+\infty}\frac{\psi'(r)}{\psi(r)}>0$, then $\lambda_1(M)>0$ (see \cite[Lemma 4.1]{bfg}). Let $(F_j)_{j=1}^n$ be a orthonormal frame on $(M,g)$, where
$F_1,...,F_{n-1}$ correspond to the spherical coordinates and
$F_n$ corresponds to the radial coordinate. The curvature
operator, the Ricci curvature and the scalar curvature can be
computed in terms of $\psi$ (see \cite[Chapter 3, 2.3]{P},
\cite[p.3]{BD}). The curvature operator is given by
$$
\mathcal R(F_i\wedge F_j)=-\frac{(\psi')^2-1}{\psi^2}F_i\wedge F_j\qquad 1\leq i,j\leq n-1\,,
$$
$$
\mathcal R(F_i\wedge F_n)=-\frac{\psi''}{\psi}F_i\wedge F_n\qquad 1\leq i\leq n-1\,.
$$
The Ricci curvature is given by
$$
Ric(F_n)=-(n-1)\frac{\psi''}{\psi}F_n\quad Ric(F_i)=\Big(- (n-2)\frac{(\psi')^2-1}{\psi^2} -\frac{\psi''}{\psi}\Big)F_i\,,
$$
where $1\leq i\leq n-1$, and the scalar curvature is
\begin{equation} \label{eq:scalar}
K=-2(n-1)\frac{\psi''}{\psi}-(n-1)(n-2)\frac{(\psi')^2-1}{\psi^2} \,.
\end{equation}
Hence, if
\begin{equation}\label{Riccibounded}
\sup_{r\geq 0} \frac{(\psi'(r))^2-1}{\psi^2(r)} <\infty \quad {\rm{and}}\quad \sup_{r\geq 0} \frac{\psi''(r)}{\psi(r)}<\infty\,,
\end{equation}
then the Ricci curvature of $(M,g)$ is bounded from below, i.e. there exists a real number $\kappa$ such that $Ric\geq -\kappa^2$. Assume furthermore that
\begin{equation}\label{sec bound}
\frac{(\psi'(r))^2-1}{\psi^2(r)}\geq 0 \quad {\rm{and}}\quad  \frac{\psi''(r)}{\psi(r)}\geq 0\,,
\end{equation}
then all the sectional curvatures are nonpositive and, by Hadamard Theorem \cite[Theorem 3.1]{do}, the
injectivity radius of $(M,g)$ is $+\infty$. Finally, if $(H)$, \eq{Riccibounded} and \eq{sec bound} hold, by \cite[Theorem 2.21]{au} the following Sobolev
embedding holds:
\begin{equation}\label{cmctemb}
H^1(M)\subset L^{2^*}(M)\,,
\end{equation}
where $2^*$ denotes the critical
Sobolev exponent given by
$$
2^*=\begin{cases}
2n/(n-2)&{\rm{if}}\,\,n\geq 3\\
\infty&{\rm{if}}\,\,n=2\,.
\end{cases}
$$


\section{Main Results}  \label{partial sym}

\subsection{Partial symmetry of least energy solutions}
\label{partial-symmetry}

Throughout this section $M$ denotes a $n$-dimensional Riemannian
model, $n\ge 2$, with the function $\psi$ in \eqref{metric}
satisfying (H). Consider the equation
\begin{equation}\label{equa}
-\Delta_g u=f(r,u)\ \ \ {\rm on}\ M\,,
\end{equation}
where $r=d(x,o)$ and $f\in C^1([0,\infty) \times \R)$ is a function satisfying the following conditions:
\begin{itemize}
\item[$(f1)$] $f(r,0)=0$ for any $r>0$ and, denoting by $f'_s(r,s)$ the derivative with respect to $s$, we have
$$
0 \le f'_s(r,s) \leq C(1+W(r)|s|^{p-1}) \qquad \text{for any } r>0 \ \text{and} \ s\in \R \, ,
$$
where $1<p\leq 2^*-1$, $C$ is a positive constant, $W: [0,\infty)\rightarrow [0,\infty)$ is such that $W\in L^\infty_{\rm loc}\big([0,\infty)\big)$ and $H^1(M)$ is continuously embedded into $L^{p+1}(M;W)$. Notice that $W$ denotes the radial function on $M$ defined by $W(x)=W\big(d(x,o)\big)$;

\item[$(f2)$] $f(r,s)s-f'_s(r,s)s^2<0$ for any $r>0$ and $s\neq 0$.


\end{itemize}
Many situations are allowed by assumptions (H) on $\psi$ and $(f1)-(f2)$ on $f$. For example, if we further assume that $\psi$ satisfies
\eqref{Riccibounded}-\eqref{sec bound}, then thanks to \eq{cmctemb} we may choose $f(r,s)=|s|^{p-1}s$. On the other hand, if we do not require any further restriction on $\psi$, we can always find a suitable function $f$ satisfying $(f1)-(f2)$, as one can see from Remark \ref{r:example} (where examples of manifolds with unbounded negative sectional curvatures are given). Thanks to $(f1)-(f2)$, we may define the action functional associated with \eqref{equa}
\begin{equation} \label{eq:Phi}
\Phi (v)= \frac{1}{2} \int_{M} |\nabla_g v |_{g}^2\,dV_{g}-\int_M F(r,v)\,dV_{g} \qquad \text{for any }  v\in H^1(M) \, ,
\end{equation}
where $F(r,s):=\int_0^s f(r,t)\, dt$, and the Nehari manifold
\begin{equation} \label{eq:Nehari}
\mathcal{N}:=\Big\{v\in H^1(M)\setminus\{0\} \, : \, \int_{M} |\nabla_g v |_{g}^2\,dV_{g}=\int_{M} f(r,v) \, v \, dV_{g} \Big\} \, .
\end{equation}
We say that $u\in \mathcal N$ is a \emph{least energy} or
\emph{ground state} solution to \eq{equa} if it achieves the following
infimum
\begin{equation} \label{eq:c}
c=\inf_{v\in \mathcal{N}} \Phi(v) \,.
\end{equation}
By $(f1)-(f2)$ the set $\mathcal N$ is a natural constraint in the sense that every constraint stationary point of the functional $\Phi$ is a ``free stationary point'' of $\Phi$ itself, i.e.
\begin{equation} \label{eq:weak-solution}
\int_M \langle \nabla_g u, \nabla_g v\rangle_g \, dV_g=\int_M f(r,u)v \, dV_g \qquad \text{for any } v\in H^1(M) \, .
\end{equation}
We also observe that assumption $(f2)$ yields $\Phi(v)>0$ for every $v\in \mathcal N$ so that by \eqref{eq:c} $c\ge 0$ and if it is achieved by some function $u\in \mathcal N$ then $c>0$, see the proof of Lemma \ref{positivity}.
\begin{defn} \label{d:foliated} A continuous function $u:M\to \R$ is foliated
Schwarz symmetric with respect to some $\Theta_0\in {\mathbb S}^{n-1}$ if the
value of $u$ at $(r,\Theta)\in M$ only depends on $r$ and
$\sigma=\arccos(\Theta \cdot \Theta_0)$, where $\cdot$ denotes the
standard scalar product in $\mathbb S^{n-1}$.
\end{defn}
We state our main symmetry result. 
\begin{thm}\label{main}
 Let $(M,g)$ be the Riemannian model defined by \eqref{metric} with $\psi$ satisfying assumption $(H)$.
 Let $f:M\times \R\rightarrow \R$ be a $C^1$ function satisfying assumptions $(f1)-(f2)$.
 Then, any least energy solution $u$ to \eqref{equa} is foliated Schwarz symmetric with respect
 to some $\Theta_0\in {\mathbb S}^{n-1}$ and strictly of one sign in $M$. Moreover, either $u$ is radial or $u$ is strictly decreasing with respect to $\sigma=\arccos(\Theta \cdot \Theta_0)\in [0,\pi]$ for any $r>0$ when $u>0$ in $M$
 (respectively strictly increasing with respect to $\sigma=\arccos(\Theta \cdot \Theta_0)\in [0,\pi]$ for any $r>0$ when $u<0$ in $M$).
\end{thm}
Theorem \ref{main} nothing says about existence of least energy
solutions. Nevertheless, when they exist, least energy solutions
$u$ to \eq{equa} are axially symmetric with respect to the axis
$\R \Theta_0 \subset \R^n$ and $u=u(r,\sigma)$. When $M={\mathbb
H}^n$ and $f=f(u)$, it is known that $u$ does not depend on $\sigma$, namely $u$ is
radial, see \cite{mancini} and \cite{adg}. For general manifolds,
a first step in this direction is made by Proposition \ref{separa}
below.

\begin{prop}\label{separa}
Under the same assumptions of Theorem \ref{main}, let $u \in H^{1}(M)\setminus\{0\}$ be a nonradial least energy solution to \eqref{equa}. If the map $t\mapsto \frac{f(x,t)}{t}$
is strictly increasing and locally Lipschitz then $u$ does not admit a representation of the type $u(r,\sigma)=R(r)h(\sigma)$ where $R:[0, +\infty)\rightarrow
\R$ and $h:[0,\pi]\rightarrow \R$.
\end{prop}

The same proof of Theorem \ref{main} yields foliated Schwarz
symmetry of least energy solutions to the Dirichlet problem
\begin{equation}
\label{dir}
\left \{ \begin{array}{ll}
-\Delta_g u=f(r,u) & \mbox{in }B_R\\
u=0 & \mbox{on }\partial B_R\,,
\end{array}
\right.
\end{equation}
where $R>0$ is fixed.
Indeed, by simply replacing $c$ in the proof of Theorem \ref{main}  with
$$c_0=\inf_{u\in \mathcal{N}_0} \Phi(u)\,,$$
where $\mathcal{N}_0:=\{u\in H_0^1(B_R)\setminus\{0\}\,:\,\int_{B_R}
|\nabla_g u |_{g}^2\,dV_{g}=\int_{B_R} f(r,u)\,u\,dV_{g} \}\,$, we get

\begin{thm}\label{main dir}
Let $(M,g)$ be the Riemannian model defined by \eqref{metric} with
$\psi$ satisfying assumption $(H)$. Let $f:M\times \R\rightarrow
\R$ be a $C^1$ function satisfying assumptions $(f1)-(f2)$.
Then, any least energy solution $u$ to \eqref{dir} is foliated Schwarz
symmetric with respect to some $\Theta_0\in {\mathbb S}^{n-1}$ and strictly of one sign in $B_R$. Moreover, either $u$ is radial or $u$ is strictly decreasing with respect to $\sigma=\arccos(\Theta \cdot \Theta_0)\in [0,\pi]$ for any $r\in (0,R)$ when $u>0$ in $B_R$ (respectively strictly increasing with respect to $\sigma=\arccos(\Theta \cdot \Theta_0)\in [0,\pi]$ for any $r\in (0,R)$ when $u<0$ in $B_R$).
\end{thm}

Let $H_{r}^1(M)$ denote the set of functions $u\in H^1(M)$ which
are radially symmetric. We define $c_r:=\inf_{u\in \mathcal{N}\cap
H_{r}^1(M)} \Phi(u)\,.$ Then, $c\leq c_r$. When $c<c_r$, least
energy solutions (if they exist) must be nonradial and by Theorem
\ref{main} they are axially symmetric. The same hold for $c_0$ and
$c_{0,r}:=\inf_{u\in \mathcal{N}_0\cap H_{r}^1(M)} \Phi(u)\,.$ If
$M=\mathbb{H}^n$ and $f(r,u)=r^{\alpha}|u|^{p-1}u$, it has been
recently proved in \cite{he} that $c_0<c_{0,r}$ as $p\rightarrow
2^*-1$. In this case, Theorem \ref{main dir} gives a
sharp information.

\begin{open} \label{o:symmetry}
{\rm When the nonlinearity depends on $r$, our symmetry result can be considered optimal in the sense explained above. One may ask if something more can be said about symmetry of least energy solutions of \eqref{equa} when $f$ only depends on $u$. In \cite{adg}, radial symmetry about some point is obtained for positive solutions of \eqref{equa} when $M=\mathbb H^n$ and $f$ is a suitable nonlinearity depending only on $u$. Here, by radial symmetry of a function $u:M\to \R$ with respect to a point $x_0\in M$, which is not necessarily the pole of the Riemannian model, we mean that $u$ is constant on geodesic balls centered at $x_0$. We leave as an open problem to show if radial symmetry about some point occurs also in more general Riemannian models, at least for ground state solutions.
 Note that the assumptions on $(M,g)$ introduced in \cite[Section 4]{adg} are too restrictive in the setting of our Riemannian models. Indeed, they would imply that $M$ has constant scalar curvature and then $\psi(r)=r$ or $\psi(r)=\alpha^{-1}\sinh(\alpha r)$, i.e. $(M,g)$ is either the euclidean space or an hyperbolic space whose scalar curvature is $-n(n-1)\alpha^2$ (see also Proposition \ref{p:constant-K}).
\endproof
}
\end{open}

\subsection{Existence of least energy solutions}\label{ex}

In Subsection \ref{partial-symmetry} we discussed partial
symmetry of least energy solutions for the equation \eqref{equa}
under the assumptions $(f1)-(f2)$. Here we address the problem of existence
of least energy solutions of \eqref{equa}. As already mentioned
in the Introduction, one way to gain existence of least
energy solutions is to assume that the manifold $M$ satisfies the
so-called weakly homogeneity condition, see \cite{marzuola}.
For completeness we recall here its definition.

\begin{defn} \label{d:weakly-hom} Let $(M,g)$ be a $n$-dimensional complete Riemannian manifold. We say that $M$ is weakly homogeneous
if there exists a group $\Gamma$ of isometries of $M$ and $D>0$ such
that for every $x,y\in M$, there exists $\gamma\in \Gamma$ such that
$d\big(x,\gamma(y)\big)\leq D$. Here $d$ represents the geodesic distance in
$M$.
\end{defn}
Assume that $(M,g)$ is a Riemannian complete weakly homogeneous manifold (not
necessarily a Riemannian model in the sense of \eqref{metric}) with bounded geometry. Proceeding as in \cite{marzuola}, where the pure power equation is dealt, we prove existence of least energy solutions
for the following class of equations
\begin{equation} \label{eq:V(x)}
-\Delta_g u=V(x)u+h(u) \qquad \text{in } M
\end{equation}
with $V$ and $h$ satisfying suitable assumptions. Similarly to what we did for \eqref{equa} we may define the Nehari manifold associated with \eqref{eq:V(x)} by
\begin{equation*}
\mathcal N:=\left\{v\in H^1(M)\setminus\{0\}: \int_M |\nabla_g v|_g^2 \, dV_g-\int_M V(x) v^2 \, dV_g-\int_M h(v)v \, dV_g \right\}
\end{equation*}
and the least energy solutions of \eqref{eq:V(x)} as the solutions $u\in \mathcal N$ of \eqref{eq:V(x)} satisfying
\begin{equation} \label{eq:Nehari-V}
\ds{\Phi(u)=c:=\inf_{v\in \mathcal{N}} \Phi(v)}
\end{equation}
where
$$\Phi(v):=\frac 12 \int_M \left[|\nabla_g v|_g^2-V(x)v^2 \right]\, dV_g-\int_M H(v)\, dV_g$$
and $H(s):=\int_0^s h(t)\, dt$ for any $s\in\R$.

\begin{prop} \label{p:existence}
Let $(M,g)$ be a Riemannian complete weakly homogeneous manifold of dimension $n\ge 2$ with positive injectivity radius and bounded curvature. Let $V\in C^0(M)$ satisfy
\begin{equation} \label{eq:bounds-V}
-\infty<V_\infty:=\inf_{x\in M} V(x)\le \sup_{x\in M} V(x)<\lambda_1(M)
\end{equation}
with $\lambda_1(M)$ as in \eqref{lambda1} and
\begin{equation} \label{eq:V(infty)}
\text{for any fixed } o\in M \quad \lim_{d(x,o)\to +\infty} V(x)=V_\infty
\end{equation}
where $d(\cdot,\cdot)$ denotes the geodesic distance in $(M,g)$. Furthermore, suppose that $h\in C^1(\R)$ satisfies
\begin{equation} \label{eq:subcrit}
h(0)=0 \, , \quad h'(0)=0 \, , \quad |h'(s)|\le C(1+|s|^{p-1}) \quad \text{for any } s\in\R
\end{equation}
with $1<p<2^*-1$,
\begin{equation} \label{eq:N-nonempty}
h(s)s-h'(s)s^2<0 \qquad \text{for any } s\neq 0\, ,
\end{equation}
and that there exists $\mu>2$ such that
\begin{equation} \label{eq:AR-h}
0<\mu H(s)\le h(s)s \qquad \text{for any } s\neq 0 \, .
\end{equation}
Then \eqref{eq:V(x)} admits a least energy solution in the sense of \eqref{eq:Nehari-V}.
\end{prop}
Notice that the geometric assumptions of Proposition \ref{p:existence} ensure the validity of the Sobolev embedding \eq{cmctemb} \cite[Theorem 2.21]{au}. We observe that condition \eq{eq:AR-h}, is the Ambrosetti-Rabinowitz condition. We refer to \cite{Sz} and references therein for possible alternative assumptions on $h$ in the euclidean setting. But in the present paper we are mainly interested in the study of
\eqref{equa} over the Riemannian model \eqref{metric}. One may ask what happens in this case if the weakly homogeneity
condition is assumed. An answer to this question is given by the
following

\begin{prop} \label{p:constant-K}
Let $(M,g)$ be a $n$-dimensional Riemannian model defined by \eqref{metric} with $n\ge 2$ and $\psi$ satisfying assumption $(H)$ and suppose that $(M,g)$
is weakly homogeneous. Then the scalar curvature is constant. Moreover if we denote by $\kappa$ this constant then only two alternatives may occur:
\begin{itemize}
\item[(i)] $\kappa<0$ and $(M,g)$ is a hyperbolic space, namely
$$
\psi(r)=\alpha^{-1} \sinh(\alpha r) \quad \text{with} \ \alpha=\sqrt{\frac{|\kappa|}{n(n-1)}};
$$

\item[(ii)] $\kappa=0$ and $(M,g)$ is the Euclidean space, namely $\psi(r)=r$.
\end{itemize}
\end{prop}
For results concerning manifolds with constant curvature see
\cite{BeBer}.
It is clear from Proposition \ref{p:constant-K} that weakly
homogeneity becomes a too restrictive condition if $(M,g)$ is the Riemannian model \eq{metric}. Therefore, we look for some
alternative conditions on $(M,g)$ and on the nonlinearity in \eqref{equa} which guarantee compactness of the embeddings of
$H^1(M)$ into suitable weighted Lebesgue spaces. More precisely, we assume that $f\in C^1([0,\infty)\times \R)$ is such that
\begin{itemize}
\item[$(f3)$] $f(r,0)=0$ for any $r>0$ and, denoting by $f'_s(r,s)$ the derivative with respect to $s$, we have
$$
-C-W(r)|s|^{p-1}\le f'_s(r,s) \le \lambda+W(r)|s|^{p-1} \qquad \text{for any } r>0 \ \text{and} \ s\in \R \, ,
$$
where $1<p<2^*-1$, $C$ is a positive constant, $\lambda\in (-C,\lambda_1(M))$ with $\lambda_1(M)$ as in \eqref{lambda1}, and $W\in L^\infty_{\rm loc}\big([0,\infty)\big)$ is a nonnegative function;

\item[$(f4)$] there exists $\mu>2$ such that
$$
0<\mu H(r,s)\le h(r,s)s \qquad \text{for any } r>0 \ \text{and} \ s\neq 0
$$
where $\lambda$ is as in $(f3)$, $h(r,s):=f(r,s)-\lambda s$ and $H(r,s):=\int_0^s h(r,t)dt$.
\end{itemize}
We are ready to state the main result of this subsection where we restrict ourselves to $n\geq 3$.


\begin{thm}
\label{compact 1}
Let $(M,g)$ be the $n$-dimensional Riemannian model defined by \eqref{metric} with $n\ge 3$ and $\psi$ satisfying assumption $(H)$. Let $f(\cdot,\cdot)$ satisfy $(f2)-(f4)$. Letting $\phi(r):=\frac{\psi(r)}{r}$ and $W$ be as in $(f3)$, we assume that
\begin{itemize}
\item[(i)] $\phi(r)\geq 1$ for every $r> 0$;
\item[(ii)] the functions $\phi^{n-1}( r) W( r)$ and $\phi^{n-3}(r)$ are nondecreasing in $[R,+\infty)$ for some $R> 0$;
\item [(iii)] $W(r)=o(\phi^{-\frac{p-1}{2}})$ as $r\rightarrow +\infty$.
\end{itemize}
Then \eqref{equa} admits a least energy solution in the sense of \eqref{eq:c}.
\end{thm}

\begin{rem} \label{r:example}
It is worth noticing that Theorem \ref{compact 1} does not work for $\psi(r)=r$ ((ii) and (iii) would yield a contradiction) while it works if $\psi(r)= \sinh r$, $\psi(r)=e^{ar}$, $\psi(r)=e^{r^a}$, $\psi(r)=r^b e^{ar}$,
$\psi(r)=r^b e^{r^a}$ for $r\geq R$, with $a\geq 1$ and $b>0$.
Namely, unbounded negative sectional curvatures are allowed. For any of the above models, as an explicit
example of $f$, take $f(r,u)=\lambda u+W(r)|u|^{p-1}u$ with $W(r)=\phi^{-\alpha}$, where
$\frac{p-1}{2}\leq \alpha \leq n-1$ and $\lambda\in (-C,\lambda_1(M))$.
\end{rem}

\begin{rem} Let $u$ be the least energy solution found in Theorem \ref{compact 1}. We observe that if in Theorem \ref{compact 1} we add the assumption $f'_s(r,s)\geq 0$ for any $r>0$ and $s\in \R$, then $u$ satisfies all the sign and symmetry properties stated in Theorem \ref{main}.
\end{rem}
Next we want to focus our attention on the condition $\lambda<\lambda_1(M)$ required in $(f3)$. If this condition is dropped, then existence of nontrivial constant sign solutions may fail as shown in the following

\begin{prop} \label{p:non-ex} Let $(M,g)$ be the $n$-dimensional Riemannian model defined by \eqref{metric} with $n\ge 3$ and $\psi$ satisfying assumption $(H)$. Let $f$ be in the form $f(r,s)=\lambda s+h(r,s)$ with $\lambda>\lambda_1(M)$ and suppose it satisfies $(f3)$ and $(f4)$. Furthermore, suppose that (i)-(iii) in Theorem \ref{compact 1} hold true with $W$ as in $(f3)$. Then \eqref{equa} does not admit any nontrivial constant sign solution $u\in H^1(M)$.
\end{prop}

In Proposition \ref{p:non-ex} we assumed that $\lambda>\lambda_1(M)$ so that one may ask what happens when $\lambda=\lambda_1(M)$. A partial answer can be found in \cite[Theorem 1.1]{mancini} where nonexistence of positive $H^1(\mathbb H^n)$-solutions is proved when $(M,g)$ coincides with the hyperbolic space $\mathbb H^n$ and $f(r,s)=\lambda_1(\mathbb H^n)s+|s|^{p-1}s$.
We conclude this section with the following open problem:

\begin{open} \label{o:existence}
{\rm In order to guarantee existence of least energy solutions we made two different kinds of assumptions: either the weakly homogeneity and the bounded geometry of $M$ in Proposition \ref{p:existence} or the compactness of the embedding $H^1(M)\subset L^{p+1}(M;W)$ in Theorem \ref{compact 1}. Both the results do not cover the case in which $M$ is a Riemannian model \eq{metric} with nonconstant curvature and $f(r,u)=|u|^{p-1}u$, $1<p<2^*-1$ in \eqref{equa}. As far as we are aware, the existence of a least energy solution for the equation
\begin{equation} \label{eq:Lane-Emden}
-\Delta_g u=|u|^{p-1}u \qquad \text{in }M\,, \quad 1<p<2^*-1
\end{equation}
is still an open problem when $M$ is a general manifold different from $\mathbb R^n$ or $\mathbb H^n$. Note that $u$ is a least energy solution of \eqref{eq:Lane-Emden} if and only if $u$ is a minimizer of
\begin{equation*}
S_p:=\inf_{v\in H^1(M)\setminus\{0\}} \frac{\int_M |\nabla_g v|^2_g \, dV_g}{\left(\int_M |v|^{p+1} \, dV_g\right)^{2/(p+1)}} \, .
\end{equation*}
We know that \eqref{eq:Lane-Emden} admits a positive radial solution which is a minimizer for
\begin{equation*}
S_{p,r}:=\inf_{v\in H^1_r(M)\setminus\{0\}} \frac{\int_M |\nabla_g v|^2_g \, dV_g}{\left(\int_M |v|^{p+1} \, dV_g\right)^{2/(p+1)}}
\end{equation*}
where $H^1_r(M)$ denotes the space of radial functions in $H^1(M)$, see \cite[Theorem 2.5]{bfg}. We ask whether the two constants $S_p$ and $S_{p,r}$ coincide. This question is strictly related with Open Problem \ref{o:symmetry}. Indeed, once proved the existence of a minimizer $u$ for $S_p$, then one way to establish the validity of the identity $S_p=S_{p,r}$ is to check whether or not $u$ is a radially symmetric function.
\endproof }
\end{open}


\section{Proof of Theorem \ref{main} and Proposition \ref{separa}}  \label{proofs}

We first prove some preliminary lemmas which we shall apply in the proof of the symmetry result.
\begin{lem}\label{positivity}
Suppose that $\psi$ satisfies $(H)$ and $f$ satisfies $(f1)-(f2)$. Then, any least energy solution to \eqref{equa} belongs to $C^2(M)$ and it is either strictly positive or strictly negative in $M$.
\end{lem}

\begin{proof}
Let $u$ be a least energy solution of \eqref{equa}. Standard Brezis-Kato type estimates \cite{Brezis-Kato} combined with elliptic regularity estimates yield $u\in C^{1,\alpha}(M)$ for any $\alpha \in (0,1)$. Then, assumption $(f1)$ combined with classical Schauder estimates yields $u\in C^2(M)$. In order to prove that $u$ does not change sign we first show that $\ds{\Phi(u)=c=\inf_{v\in \mathcal N} \Phi(v)>0}$. Indeed by $(f2)$ and a simple integration by parts we infer $2F(r,s)-f(r,s)s<0$ for any $r>0$ and $s\neq 0$, and hence since $u\in \mathcal N$ we obtain
\begin{equation*}
c=\Phi(u)=-\frac 12\int_M [2F(r,u)-f(r,u)u]\, dV_g>0 \, .
\end{equation*}
To show that $u$ is positive or negative, let $u^+:=\max\{u,0\}$ and $u^-:=-\min\{u,0\}$ denote the positive and negative part of $u$. Suppose that $u^+,u^-$ are not identically equal to zero in $M$. Testing \eqref{eq:weak-solution} with
$u^+$ and $u^-$, one gets $u^+,-u^-\in \mathcal{N}$. Then,
$$
c=\Phi(u)=\Phi(u^+)+\Phi(-u^-)\geq 2c,
$$
a contradiction with $c=\Phi(u)>0$. Hence, we have $u\geq 0$ or $u\leq 0$. An application of the strong maximum principle yields the strict inequalities, see \cite{Gilb}.
\end{proof}
For every $\Theta_0 \in {\mathbb S}^{n-1}$, let
$$H_{\Theta_{0}}:=\{(r, \Theta)\,: r\geq 0,\quad \Theta \cdot \Theta_0\geq 0 \}\,,$$
and $\cH :=\{H_{\Theta_{0}}\,:\Theta_0 \in {\mathbb S}^{n-1}\}$.  For $\Theta_0 \in {\mathbb S}^{n-1}$,
we define $p_{H_{\Theta_{0}}}: M \to M$ by $p_{H_{\Theta_{0}}}(r,\Theta)=(r,\Theta-2(\Theta\cdot \Theta_0)\Theta_0)$. For simplicity, we also put $H=H_{\Theta_{0}}$ and $\overline{x}=p_{H}(x)$ for $x \in M$ when the underlying $H_{\Theta_{0}}$ is understood. For a measurable
function $v: M \to \R$, we define the {\em polarization} $v_H$ of $v$
relative to $H$ by
\begin{equation}
  \label{eq:def-polarization}
v_H(x) = \left \{
  \begin{aligned}
    &\max\{v(x),v(\overline{x})\},&&\qquad x \in H,\\
    &\min\{v(x),v(\overline{x})\},&& \qquad x \in M\setminus H.
  \end{aligned}
\right.
\end{equation}

\begin{lem}\label{strict-alt}
Suppose that $\psi$ satisfies $(H)$ and $f$ satisfies $(f1)-(f2)$. Let $u$ be a (positive) least energy solution to \eqref{equa} and $H \in \cH$. Then, one of the following is true:
\begin{itemize}
\item[(i)] $u > u \circ p_H$ in $\innt(H)$,\\
\item[(ii)] $u < u \circ p_H$ in $\innt(H)$,\\
\item[(iii)] $u \equiv u \circ p_H$ in $M$.
\end{itemize}
\end{lem}

\begin{proof} Let us define $M^+:=\{x\in M: u(x)\ge u_H(x)\}$ and $M^-:=\{x\in M: u(x)<u_H(x)\}$. Similarly to \cite[Proposition 2.3]{Smets-Willem}, by \eqref{metric} we have
\begin{equation} \label{eq:grad-uH}
\nabla_g u_H(x)=
\begin{cases}
\nabla_g u(x) & \qquad \text{a.e. on } (H\cap M^+)\cup [(M\setminus H)\cap M^-] \\[7pt]
\nabla_g u(p_H(x)) & \qquad \text{a.e. on } (H\cap M^-) \cup [(M\setminus H)\cap M^+] \, .
\end{cases}
\end{equation}
By \eqref{metric}, \eqref{eq:grad-uH} and the change of variable $y=p_H(x)$ we obtain
\begin{align} \label{eq:L2-grad}
& \int_M |\nabla_g u_H|_g^2 \, dV_g =\int_{H\cap M^+} |\nabla_g u|_g^2 \, dV_g+
\int_{(M\setminus H)\cap M^-} |\nabla_g u|_g^2 \, dV_g \\
\notag & \qquad +\int_{H\cap M^-} |\nabla_g u(p_H(x))|_g^2 \, dV_g(x)+\int_{(M\setminus H)\cap M^+} |\nabla_g u(p_H(x))|_g^2 \, dV_g(x) \\
\notag & \quad =\int_{H\cap M^+} |\nabla_g u|_g^2 \, dV_g+\int_{(M\setminus H)\cap M^-} |\nabla_g u|_g^2 \, dV_g \\
\notag & \qquad +\int_{(M\setminus H)\cap M^+} |\nabla_g u(y)|_g^2 \, dV_g(y)+\int_{H\cap M^-} |\nabla_g u(y)|_g^2 \, dV_g(y)\\
\notag &=\int_M |\nabla_g u|_g^2 \, dV_g \, .
\end{align}
In the same way we also obtain
\begin{equation} \label{eq:fF}
 \int_M F(r,u_H)\,dV_{g}=\int_M F(r,u)\,dV_{g} \,,\quad
 \int_M f(r,u_H)u_H\,dV_{g}=\int_M f(r,u)u\,dV_{g} \, .
\end{equation}
Combining \eqref{eq:L2-grad} and \eqref{eq:fF} we deduce that if $u\in \mathcal{N}$ is a minimizer of $\Phi$ on $\mathcal N$, then also $u_H\in \mathcal{N}$ is a minimizer of $\Phi$ on $\mathcal N$. Therefore, it follows that $u_H$ is a (positive) least energy solution
of \eqref{equa}, see Lemma \ref{l:relation-PS} for more more details.
Following \cite{Bartsch}, we consider
$$w:H\rightarrow \R\,,\quad  w:=|u-u \circ p_H|=2u_H-(u+u \circ p_H)\,.$$
By Lemma \ref{positivity} we have that $w\in C^2(M)$. Furthermore, from $(f1)$ and the fact that $p_H$ is an isometry in $M$, we deduce
\begin{align*}
& -\Delta_g w= -2\Delta_g u_H+\Delta_g u+\Delta_g(u\circ p_H)=2 f(r,u_H)-f(r,u)+(\Delta_g u)\circ p_H \\
& \quad
=[f(r,u_H)-f(r,u)]+[f(r,u_H)-f(r,u \circ p_H)]\geq 0\ \ \ {\rm on}\ H  \, .
\end{align*}

For every $x\in H$, let $B_x$ be a ball such that $x\in \overline{B}_x\subset H$. By the strong maximum principle for elliptic operators, see \cite{Gilb}, we deduce that either $w\equiv 0$ in $ \overline{B}_x$ or $w>0$ in $ B_x$ and hence we conclude that either $u\equiv u \circ p_H$ on $H$ or $|u-u \circ p_H|>0$ on $\text{int}(H)$.
\end{proof}

For any $\Theta_0 \in {\mathbb S}^{n-1}$ we now define
$$\cH(\Theta_0):= \{H \in \cH:\, \Theta_0 \in \innt(H)\}.$$
We will prove Theorem~\ref{main} with the help of the following
characterization that can be deduced by \cite[Proposition 2.4]{weth}.

\begin{lem}\label{sobol2}
Let $\Theta_0 \in {\mathbb S}^{n-1}$. A continuous function $v: M \to \R$ is foliated Schwarz
symmetric with
respect to $\Theta_0$ if and only if $v=v_H$ for every $H\in \cH(\Theta_0)$.
\end{lem}

\bigskip
\textbf{Proof of Theorem \ref{main}.} It is not restrictive assuming that $u$ is a positive least energy solution to \eq{equa} since otherwise one may define $\tilde u=-u$ and obtain a positive least energy solution of
$$
-\Delta_g \tilde u=\widetilde f(r,\tilde u) \qquad \text{in } M
$$
where $\widetilde f(r,s):=-f(r,-s)$ defined for any $r>0$ and $s\in \R$ satisfies assumptions $(f1)-(f2)$.

Take $\Theta_0 \in \mathbb S^{n-1}$ such that
$$
u(1,\Theta_0)= \max \{u(1,\Theta)\,: \Theta \in {\mathbb S}^{n-1} \} \, .
$$
Then, for any $H\in \mathcal H(\Theta_0)$ there exists $y\in \text{int}(H)$ such that $u(y)\ge u(p_H(y))$ and hence
by Lemma \ref{strict-alt} we deduce that only case (i) or (iii) in the same lemma may occur; in both situations this yields
$$
u=u_H \qquad \text{for every $H \in \cH(\Theta_0)$.}
$$
Hence, $u$ is foliated Schwarz symmetric with respect to $\Theta_0$ by
Lemma \ref{sobol2}. Therefore we can write $u=u(r,\sigma)$, where $\sigma= \arccos(\Theta \cdot \Theta_0)$. It remains to prove
that
\begin{equation}
\label{eq:alternative}
\text{either $u$ is radial, or $u(r,\sigma)$ is strictly
decreasing in $\sigma \in (0,\pi)$ for $r>0$.}
\end{equation}
We follow the argument in \cite[p.204]{girao-weth}. We already know that no
half-space $H \subset \cH(\Theta_0)$ satisfies
property (ii) of Lemma \ref{strict-alt}. Moreover, if property (i) of this
lemma holds for all half-spaces $H \subset \cH(\Theta_0)$, then $u(r,\sigma)$ is strictly
decreasing in $\sigma \in (0,\pi)$ for every $r>0$. It remains to
consider the case where property (iii) of Lemma~\ref{strict-alt} holds
for some $H_0 \subset \cH(\Theta_0)$. Let $0<\sigma_0<\pi/2$ be the angle
formed by $\Theta_0$ and the hyperplane $\partial H_0$. Let
$\Theta_1= p_{H_0}(\Theta_0)$. Then $\arccos(\Theta_1 \cdot \Theta_0)=2\sigma_0$. Moreover,
(iii) implies that $u(r,\Theta_1)=u(r,\Theta_0)$ for $r>0$. Since $u$ is nonincreasing in the angle
$\sigma \in (0,\pi)$, we conclude that $u(r,\sigma)= u(r,0)$ for all
$\sigma \le 2 \sigma_0$. From Lemma \ref{strict-alt} we then deduce that
(iii) holds for all $H \subset \cH(\Theta_0)$ for
which the angle between $\Theta_0$ and $H$ is less then $2 \sigma_0$. Then, by the
same argument as before, $u(r,\sigma)= u(r,0)$ for all $\sigma \le \min\{4 \sigma_0,\pi\}$. Arguing
successively, in a finite number of steps
we obtain $u(r,\sigma)= u(r,0)$ for all $\sigma \le \pi$. This shows that
$u$ is radial and completes the proof of \eqref{eq:alternative}. \hfill $\Box$\par\vskip3mm

\bigskip

\textbf{Proof of Proposition \ref{separa}.} \label{s:separabili} We exploit the idea of \cite[Remark 6.3]{girao-weth}. It is not restrictive to assume $u>0$. By contradiction, let $u(r,\sigma)=R(r)h(\sigma)>0$. Since we are also assuming that $u$ is not radial, then by Theorem \ref{main} we deduce that $h$ is strictly decreasing in $[0,\pi]$. Few computations yield
\begin{align*}
& -\frac{\psi^2(r)}{\psi^{n-1}(r)R(r)}\frac{\partial}{\partial r}\left( R'(r)\psi^{n-1}(r)\right)-\frac{1}{\sin^{n-2}(\sigma)h(\sigma)}\frac{\partial}{\partial \sigma}
\left( h'(\sigma)\sin^{n-2}(\sigma)\right) \\
& \qquad = \psi^2(r)\frac{f(r, R(r)h(\sigma))}{R(r)h(\sigma)} \, .
\end{align*}
If $0\le \sigma_1<\sigma_2\le \pi$, then $h(\sigma_1)>h(\sigma_2)$ and hence for any $\overline r>0$ there exists $L>0$ such that
\begin{align*}
& 0<b(\sigma_1)-b(\sigma_2)= \psi^2(r)\left(\frac{f(r,R(r)h(\sigma_2))}{ R(r)h(\sigma_2)}-\frac{f(r,R(r)h(\sigma_1))}{ R(r)h(\sigma_1)}\right) \\
& \qquad \leq \psi^2(r) L R(r)(h(\sigma_1)-h(\sigma_2)) \qquad \text{for any } r\in (0,\overline r) \, ,
\end{align*}
where $b(\sigma):=\frac{1}{\sin^{n-2}(\sigma)h(\sigma)}\frac{\partial}{\partial \sigma}\left( h'(\sigma)\sin^{n-2}(\sigma)\right)$.
Namely, there exists $C>0$ such that $R(r)\geq C \psi^{-2}(r)$ for any $r\in (0,\overline r)$. This gives a contradiction as $r\rightarrow 0$ since $u$ is a classical solution of \eqref{equa} in view of Lemma \ref{positivity}. \hfill $\Box$\par\vskip3mm

\section{Proof of Theorem \ref{compact 1} and Proposition \ref{p:non-ex}} \label{s:existence-1}

We first state some basic facts. Let $w,v_0,v_1:\mathbb R^n\rightarrow (0,\infty)$ be measurable locally bounded functions. For $q\geq 1$, we denote by $L^q(\mathbb R^n;w)$ the set of all measurable functions $u$ on $\mathbb R^n$ such that
$$
\int_{\mathbb R^n} |u(y)|^q w(y) \,dy<\infty\,,
$$
and by $W^{1,2}(\mathbb R^n;v_0,v_1)$ the set of all functions $u\in L^2(\mathbb R^n;v_0)$ such that $|\nabla u|\in  L^2(\mathbb R^n;v_1)$. Then, we associate to Sobolev and $L^q$ spaces on the model manifold $M$ suitable weighted Sobolev and $L^q$ spaces on $\mathbb R^n$.

\begin{lem} \label{norme}
Set $\phi(r):=\frac{\psi(r)}{r}$ and suppose that $\phi(r)\geq 1$ for all $r>0$. For every $n\geq 3$, the following hold:
\begin{itemize}
\item[(i)] $\|u\|_{L^{q}(M)} =\|u\|_{L^{q}(\mathbb R^n; \phi^{n-1})}\qquad 1\leq q<\infty$;
\item[(ii)] $\|u\|_{H^1(M)}\geq \|u\|_{  W^{1,2}(\mathbb R^n;\phi^{n-1}, \phi^{n-3})    }\,.$
\end{itemize}
\end{lem}
\begin{proof}
The equality (i) is an immediate consequence of the formula for the Riemannian measure in polar coordinates. To prove (ii) notice that
$$
\begin{aligned}
\int_{M} |\nabla_g u |_{g}^2\,dV_{g}\,&=\,\int_{0}^{+\infty}\int_{{\mathbb S}^{n-1}} [u_r^2+ \, \frac{1}{\psi^{2}}|\nabla_{{\mathbb S}^{n-1}} u|^2]\,\psi^{n-1}\, dr\,d\Theta\\
&\geq \int_{0}^{+\infty}\int_{{\mathbb S}^{n-1}} [u_r^2+ \, \frac{1}{r^{2}}|\nabla_{{\mathbb S}^{n-1}} u|^2]\,r^{n-1}\,   \big(\frac{\psi}{r}\big)^{n-3}  dr\,d\Theta\\
&=\|\nabla u\|^2_{ L^2(\mathbb R^n; \phi^{n-3})   }\,.
\end{aligned}
$$
\end{proof}

We will exploit the following

\begin{prop}\label{embedding}
Let $n\geq 3$, $1<p\leq 2^*-1$ and $w,v_0,v_1:\mathbb R^n\rightarrow (0,\infty)$ be locally bounded radial weights on $\mathbb R^n$, i.e. measurable and a.e. positive functions in $\R^n$. Suppose there exists $R>0$, a positive measurable function $\delta$ defined on $\mathbb R^n\setminus B(0,R)$, $c_{\delta}\geq 1$ and $k_0>0$ such that
\begin{itemize}
\item[(i)] $k_0 \left(\frac{v_1(x)}{v_0(x)}\right)^{1/2}\leq \delta(x)\leq \frac{|x|}{3}\qquad for\, a.e.\, x\in \mathbb R^n\setminus B(0,R)$;
\item[(ii)] $c_{\delta}^{-1}\leq \frac{\delta(y)}{\delta(x)}\leq c_{\delta}\qquad for \,a.e.\, x\in \mathbb R^n\setminus B(0,R)\text{ and }y\in B(x,\delta(x))\,.$
\end{itemize}
Assume furthermore that there exist two positive measurable functions $b_0,b_1$ defined on $\mathbb R^n\setminus B(0,R)$
such that
\begin{equation}\label{wv}
w(y)\leq b_0(x)\,,\, b_1(x)\leq v_1(y)\, for \,a.e.\, x \, \in \mathbb R^n\setminus B(0,R)\text{ and } y\in B(x,\delta(x))\,.
\end{equation}
If
\begin{equation}\label{limitezero}
\lim_{m\rightarrow +\infty}  \sup_{|x|\geq m}  \frac{b_0(x)^{\frac{1}{p+1}}}{b_1(x)^{1/2}}\,\delta(x)^{  \frac{n}{p+1}-\frac{n}{2}+1 } =0\,,
\end{equation}
then the embedding $W^{1,2}(\mathbb R^n;v_0,v_1)\subset L^{p+1}(\mathbb R^n;w)$ is continuous and if $1<p< 2^*-1$, the embedding is also compact.
\end{prop}
\begin{proof}
The proof follows by slightly modifying the proof of \cite[Theorem 18.7]{kufner}.
\end{proof}

By combining Lemma \ref{norme} with Proposition \ref{embedding} we get
\begin{lem} \label{our embedding}
Let $n\geq 3$, $1<p\leq 2^*-1$ and $(M,g)$ be the $n$-dimensional Riemannian model defined by \eqref{metric} with $\psi$ satisfying assumption $(H)$. Let $W\in L^\infty_{\rm loc}\big([0,\infty)\big)$ be a nonnegative function such that (i)-(iii) of Theorem \eqref{compact 1} hold. Then, the embedding $H^1(M)\subset L^{p+1}(M;W)$ is continuous and if $1<p< 2^*-1$, the embedding is also compact.
\end{lem}

\begin{proof}
By Lemma \ref{norme} (ii) we have that $H^1(M)\subset W^{1,2}(\mathbb R^n;\phi^{n-1},  \phi^{n-3} )$. On the other hand, by applying Proposition \ref{embedding} with $R=3$ and
\begin{align*}
& \delta(x)=\phi^{-1}(|x|) \, ,   \quad w(x)=\phi^{n-1}(|x|) \, W(|x|)\, , \\
& v_0(x)=v_0(|x|)=\phi^{n-1}(|x|) \, , \quad v_1(x)=v_1(|x|)= \phi^{n-3}(|x|) \, , \\
& b_0(x)=\phi^{n-1}(|x|+\delta(x)) W(|x|+\delta(x)) \, , \quad b_1(x)=\phi^{n-3}(|x|-\delta(x))\,,
\end{align*}
we deduce that the embedding $W^{1,2}(\mathbb R^n;\phi^{n-1},  \phi^{n-3}  )\subset L^{p+1}(\mathbb R^n; \phi^{n-1}W)$ is continuous and if $1<p< 2^*-1$, it is also compact. This, combined with Lemma \ref{norme} (i), proves the statement.
\end{proof}

Let us consider the Nehari manifold $\mathcal N$ defined in \eqref{eq:Nehari} and let us define the functional $I:H^1(M)\to \R$ by
\begin{equation*}
I(u):=\langle \Phi'(u),u\rangle \qquad \text{for any } u\in H^1(M) \, ,
\end{equation*}
where $\Phi'$ denotes the Fr\'echet derivative of $\Phi$. In this way we have $\mathcal N=I^{-1}(0)$ and thanks to condition $(f2)$ we also have $I'(u)\neq 0$ for any $u\in \mathcal N$. In the sequel we use the notation
\begin{equation} \label{eq:sfera-H1}
S:=\{v\in H^1(M): \|v\|_{H^1}=1\}   \, .
\end{equation}
In the next lemma we prove some properties of the Nehari manifold exploiting in the proofs the results contained in \cite{Sz}.

\begin{lem}  \label{l:nehari-1}
Suppose that all the assumptions of Theorem \ref{compact 1} are satisfied. Then we have
\begin{itemize}
\item[(i)] the functional $\Phi$ introduced in \eqref{eq:Phi} is well defined on $H^1(M)$ and it is of class $C^1$ over $H^1(M)$;

\item[(ii)] the set $\mathcal N$ defined in \eqref{eq:Nehari} is nonempty and for any $v\in H^1(M)\setminus \{0\}$ there exists a unique $t_v>0$ such that $t_v v\in \mathcal N$;

\item[(iii)] if for any $v\in H^1(M)\setminus \{0\}$ we define $\phi_v(t):=\Phi(tv)$ for any $t>0$, then $\phi_v\in C^1(0,\infty)$ and it satisfies
$\phi_v'(t)>0$ for any $t\in (0,t_v)$, $\phi_v'(t_v)=0$ and $\phi_v'(t)<0$ for any $t>t_v$;

\item[(iv)] the following estimates hold true:
\begin{equation*}
\inf_{v\in\mathcal N} \Phi(v)>0 \, , \qquad  \inf_{v\in \mathcal N} \|v\|_{H^1}>0\, , \qquad \sup_{v\in K} \|t_v v\|_{H^1}<+\infty
\end{equation*}
for any $K\subset S$ compact in $S$ with $S$ as in \eqref{eq:sfera-H1};

\item[(v)] the map
$$
m:S\to \mathcal N \, ,\quad m(v):=(t_v v)_{|S}
$$
is a homeomorphism from $S$ onto $\mathcal N$ such that $m^{-1}(v)=v/\|v\|_{H^1}$ for any $v\in S$.
\end{itemize}
\end{lem}

\begin{proof} {\bf (i)} By $(f3)$ we have that
\begin{equation} \label{eq:est-f-F}
|f(r,s)|\le C\left(|s|+\tfrac{W(r)}{p}|s|^p\right)  \quad \text{and} \quad |F(r,s)|\le C\left(\tfrac{s^2}2+\tfrac{W(r)}{p(p+1)} |s|^{p+1}\right) \quad
\end{equation}
for any $r>0$ and $s\in \R$. Combining these two estimates with Lemma \ref{our embedding} the proof of (i) follows in a standard way.

{\bf (ii)-(iii)} Consider an arbitrary nontrivial function $v\in H^1(M)$ and consider the function $\phi_v(t):=\Phi(tv)$ defined for any $t>0$.
By \eq{eq:est-f-F}, Lemma \ref{our embedding} and the fact that $\lambda<\lambda_1(M)$ we obtain
\begin{equation} \label{eq:MP-structure}
\Phi(u)\ge C_1 \|u\|_{H^1}^2-C_2 \|u\|_{H^1}^{p+1} \qquad \text{for any } u\in H^1(M)
\end{equation}
for some suitable constants $C_1,C_2>0$. This shows that $\Phi$ is strictly positive in a neighborhood of $u=0$ and, in turn, also the function $\phi_v$ is positive in a right neighborhood of $t=0$.
On the other hand by $(f4)$ and integration it follows
\begin{equation*}
H(r,s) \ge \Upsilon(r)|s|^\mu   \qquad \text{for any } r>0 \ \text{and} \ |s|\ge \sigma
\end{equation*}
where $\sigma>0$ is fixed arbitrarily and $\Upsilon(r):=\min\left\{\tfrac{H(r,\sigma)}{\sigma^\mu},\tfrac{H(r,-\sigma)}{\sigma^\mu}\right\}>0$.
Therefore, for $v\in H^1(M)\setminus\{0\}$ we have
\begin{align} \label{eq:Phi(tv)}
\phi_v(t)&=\Phi(tv)\le \frac{1}2 \left(\int_M |\nabla_g v|_g^2 \, dV_g-\lambda \int_M v^2 dV_g\right)t^2 \\
\notag & -\int_{\{|tv|<\sigma\}} H(r,tv)\, dV_g
-\left(\int_{\{|tv|\ge \sigma\}} \Upsilon(r)|v|^\mu dV_g\right)t^\mu\to -\infty
\end{align}
as $t\to +\infty$ being $\mu>2$ and $\int_{\{|tv|<\sigma\}} H(r,tv)\, dV_g=O(t^2)$ as $t\to +\infty$ in view of \eqref{eq:est-f-F} and since, by (i) and (iii) in Theorem \ref{compact 1}, $W$ is globally bounded. Summarizing $\phi_v\in C^1(0,\infty)$, $\phi_v(0)=0$, $\phi_v$ is positive in a right neighborhood of $t=0$ and eventually negative as $t\to +\infty$ and hence there exists $t_v>0$ such that $\phi_v'(t_v)=0$. This implies $t_v v\in \mathcal N$ thus proving that $\mathcal N\neq \emptyset$. It remains to prove that $t_v$ is the unique stationary point of $\phi_v$ in $(0,+\infty)$. To this purpose, for any $v\in H^1(M)\setminus\{0\}$ we define the function $\chi_v(t):=I(tv)$ for any $t>0$ in such a way that stationary points of $\phi_v$ coincide with zero points of $\chi_v$. By $(f2)$ we have
\begin{align}  \label{eq:C1-manifold}
\langle I'(u),u \rangle &=2\int_M |\nabla_g u|_g^2 \, dV_g-\int_M f'_s(r,u)u^2 dV_g-\int_M f(r,u)u \, dV_g\\
\notag & = \int_M [f(r,u)u-f'_s(r,u)u^2]\, dV_g<0 \qquad \text{for any } u\in \mathcal N   \, .
\end{align}
This implies that if $v\in H^1(M)\setminus \{0\}$ and $t>0$ is such that $\chi_v(t)=0$ then $tv\in \mathcal N$ and $\chi_v'(t)=\frac 1t \langle I'(tv),tv\rangle<0$ thus proving that $\chi_v$ admits a unique positive zero point. This completes the proof of (ii) and (iii).

{\bf (iv)} The first two inequalities in (iv) follows from \eqref{eq:MP-structure} and (iii). In order to prove the third inequality let us proceed by contradiction by supposing that there exists a sequence $\{v_k\}\subset K$ such that $v_k\to v\in K$ and $t_{v_k} \to +\infty$ as $k\to +\infty$. For simplicity in the rest of the proof we write $t_k$ in place of $t_{v_k}$. Proceeding as in \eqref{eq:Phi(tv)}, as $k\to +\infty$ we obtain
\begin{align}  \label{eq:Phi(tkvk)}
\Phi(t_k v_k) & \le \frac{1}2 \left(\int_M |\nabla_g v_k|_g^2 \, dV_g-\lambda \int_M v_k^2 dV_g\right)t_k^2 \\
\notag & -\int_{\{|t_k v_k|<\sigma\}} H(r,t_k v_k)\, dV_g
-\left(\int_{\{|t_k v_k|\ge \sigma\}} \Upsilon(r)|v_k|^\mu dV_g\right)t_k^\mu\to -\infty
\end{align}
since
\begin{align*}
& \lim_{k\to +\infty}\left(\int_M |\nabla_g v_k|_g^2 \, dV_g-\lambda \int_M v_k^2 \, dV_g\right)=\int_M |\nabla_g v|_g^2 \, dV_g-\lambda \int_M v^2 dV_g \, , \\[8pt]
& \int_{\{|t_k v_k|<\sigma\}} H(r,t_k v_k)\, dV_g=O(t_k^2) \qquad \text{as } k\to +\infty \, , \quad \text{and} \\[8pt]
& \liminf_{k\to +\infty} \int_{\{|t_k v_k|\ge \sigma\}} \Upsilon(r)|v_k|^\mu dV_g \ge \int_M \Upsilon(r)|v|^\mu dV_g>0
\end{align*}
where we used the strong convergence $v_k\to v$ in $H^1(M)$ and the Fatou Lemma. The proof is complete since \eqref{eq:Phi(tkvk)} contradicts the fact that $\phi$ is positive on $\mathcal N$.

{\bf (v)} It follows immediately applying Proposition 8 in \cite{Sz}. Indeed assumptions ($A_2$)-($A_3$) in \cite{Sz} are an immediate consequence of (iii) and (iv).
\end{proof}

Next we define $\Psi:S\to \R$ by putting $\Psi(v):=\Phi(m(v))$ for any $v\in S$. We say that a sequence $\{v_k\}\subset S$ is a Palais-Smale sequence for $\Psi$ if $\{\Psi(v_k)\}$ is bounded and $\Psi'(v_k)\to 0$ as $k\to +\infty$, i.e.
$$
\lim_{k\to +\infty} \ \sup_{w\in T_{v_k}S} \frac{|\langle \Psi'(v_k),w\rangle|}{\|w\|_{H^1}}=0
$$
where $T_{v_k}S$ denotes the tangent space to $S$ at $v_k$. In the next result we show that there is a strict relationship between Palais-Smale sequences for $\Psi$ and $\Phi$.

\begin{lem} \label{l:relation-PS}
Suppose that all the assumptions of Theorem \ref{compact 1} are satisfied. Then we have
\begin{itemize}
\item[(i)] if $\{u_k\}\subset S$ is Palais-Smale sequence for $\Psi$ then $\{m(u_k)\}$ is a Palais-Smale sequence for $\Phi$;

\item[(ii)] if $u\in S$ is a critical point for $\Psi$ then $m(u)$ is a nontrivial critical point for $\Phi$;

\item[(iii)] if $u\in \mathcal N$ is a constrained critical point for $\Phi$ then $u$ is a critical point for $\Phi$.
\end{itemize}
\end{lem}

\begin{proof} The proof of (i)-(ii) follows immediately applying Corollary 10 in \cite{Sz}. Indeed assumptions ($A_2$)-($A_3$) in \cite{Sz} are an immediate consequence of (iii) and (iv) in Lemma \ref{l:nehari-1}. In order to prove (iii) we observe that if $u\in \mathcal N$ is a constrained critical point for $\Phi$ then by the Lagrange Multiplier Method there exists $\kappa \in \R$ such that $\Phi'(u)=\kappa I'(u)$. In particular
$$
0=I(u)=\langle \Phi'(u),u\rangle=\kappa \langle I'(u),u\rangle
$$
and by \eqref{eq:C1-manifold} we know that $\langle I'(u),u\rangle\neq 0$ thus proving that $\kappa=0$ and, in turn, that $\Phi'(u)=0$.
\end{proof}


Then we prove that the functional $\Phi$ satisfies the Palais-Smale condition.

\begin{lem} \label{l:PS-validity}
Suppose that all the assumptions of Theorem \ref{compact 1} are satisfied.
Then the functional $\Phi$ satisfies the Palais-Smale condition.
\end{lem}

\begin{proof} Let $\{v_k\}$ be a Palais-Smale sequence for $\Phi$. We first prove that $\{v_k\}$ is bounded in $H^1(M)$.
Since $\{\Phi(v_k)\}$ is bounded and $\Phi'(v_k)\to 0$ is $(H^1(M))'$ as $k\to +\infty$, then by $(f3)-(f4)$ we obtain
\begin{align*}
C_1& +o(\|v_k\|_{H^1}) \ge \mu \Phi(v_k)-\langle \Phi'(v_k),v_k\rangle \\
& =\frac{\mu-2}2 \left(\int_M |\nabla_g v_k|_g^2\, dV_g-\lambda \int_M v_k^2 \, dV_g\right)\\
&+\int_M [h(r,v_k)v_k-\mu H(r,v_k)]\, dV_g
\ge C_2 \|v_k\|_{H^1}^2
\end{align*}
for some suitable positive constants $C_1,C_2$. This proves boundedness of $\{v_k\}$.
Then, up to a subsequence, we may assume that there exists $v\in H^1(M)$ such that $v_k \rightharpoonup v$ weakly in $H^1(M)$. Hence, by Lemma \ref{our embedding} $v_k\to v$ in $L^{p+1}(M;W)$.\par
By $(f3)-(f4)$ we deduce that
$$
0\le h(r,s)s\le \frac{W(r)}p|s|^{p+1} \qquad \text{for any } r>0 \ \text{and} \ s\in \R
$$
and by the strong convergence $v_k\to v$ in $L^{p+1}(M;W)$ we obtain
\begin{align} \label{eq:passaggio}
& \lim_{k\to +\infty} \int_M h(r,v_k)v_k \, dV_g=\int_M h(r,v)v \, dV_g \, , \\
\notag & \lim_{k\to +\infty} \int_M h(r,v_k)w \, dV_g=\int_M h(r,v)w \, dV_g \qquad \text{for any } w\in H^1(M) \, .
\end{align}
By \eqref{eq:passaggio} and the fact that $\{v_k\}$ is a weakly convergent Palais-Smale sequence, we infer
\begin{align} \label{eq:conv-norms}
& \lim_{k\to +\infty} \left(\int_M |\nabla_g v_k|_g^2 \, dV_g-\lambda\int_M v_k^2 \, dV_g\right)\\
&=\lim_{k\to +\infty} \left(\langle \Phi'(v_k),v_k\rangle+\int_M h(r,v_k)v_k \, dV_g\right) \\
\notag \qquad & =\int_M h(r,v)v \, dV_g=\lim_{k\to +\infty} \left(\langle \Phi'(v_k),v\rangle+\int_M h(r,v)v \, dV_g\right)\\
\notag \qquad & =\int_M |\nabla_g v|_g^2 \, dV_g-\lambda\int_M v^2 \, dV_g \, .
\end{align}
Since $\lambda<\lambda_1(M)$ then the map $w\mapsto \left(\int_M |\nabla_g w|_g^2 \, dV_g-\lambda\int_M w^2 \, dV_g\right)^{1/2}$ is an equivalent norm in $H^1(M)$ and hence \eqref{eq:conv-norms}, together with the weak convergence $v_k \rightharpoonup v$, yields $v_k \to v$ strongly in $H^1(M)$. This completes the proof of the lemma.
\end{proof}

\textbf{Proof of Theorem \ref{compact 1}.}
Let $\{w_k\}\subset \mathcal N$ be a minimizing sequence for $\Phi$ on $\mathcal N$ and for any $k\in \N$ define $\widehat w_k:=w_k/\|w_k\|_{H^1}$. Then $\{\widehat w_k\}$ is a minimizing sequence for $\Psi$ on $S$. By a classical argument based on the Ekeland Variational Principle we deduce that there exists a sequence $\{v_k\}\subset S$ such that $\Psi(v_k)\le \Psi(\widehat w_k)$, $\|\widehat w_k-v_k\|_{H^1}\to 0$ and $\Psi'(v_k)\to 0$ as $k\to +\infty$ (see for example \cite[Corollary A3]{costa} for more details). This means that $\{v_k\}$ is a Palais-Smale sequence for $\Psi$ on $S$ and hence by Lemma \ref{l:relation-PS} (i) it follows that
$\{u_k\}\subset \mathcal N$ is a Palais-Smale sequence for $\Phi$ where we put $u_k:=m(v_k)$ for any $k\in \N$. By Lemma \ref{l:PS-validity} we deduce that up to a subsequence $\{u_k\}$ strongly converges in $H^1(M)$ to a function $u\in H^1(M)$. This proves that $u\in \mathcal N$, $u$ is a minimizer for $\Phi$ on $\mathcal N$ and by Lemma \ref{l:relation-PS} (iii) $u$ is a nontrivial critical point for $\Phi$. \hfill $\Box$\par\vskip3mm



\bigskip

\textbf{Proof of Proposition \ref{p:non-ex}.}  We follow closely the argument used in the proof of Theorem 1.1 in \cite{mancini}. For completeness we give here the details. First of all using a standard truncation argument one can show that the space $C^\infty_c(M)$ of $C^\infty$ functions with compact support is dense in $H^1(M)$. Then we show that $\ds{\lim_{R\to +\infty} \lambda_1(B_R)=\lambda_1(M)}$ with
$$
\ds{\lambda_1(B_R):=\inf_{v\in H^1_0(B_R)\setminus \{0\}} \frac{\int_{B_R} |\nabla_g v|_g^2 \, dV_g}{\int_{B_R} v^2 \, dV_g}} \, .
$$
Clearly the map $R\mapsto \lambda_1(B_R)$ is nonincreasing and $\lambda_1(M)\le \lambda_1(B_R)$ for any $R>0$ and moreover by density of $C^\infty_c(M)$ in $H^1(M)$ it follows that for any $\eps>0$ there exists $v\in C^\infty_c(M)$ such that
$$
\frac{\int_{M} |\nabla_g v|_g^2 \, dV_g}{\int_{M} v^2 \, dV_g}<\lambda_1(M)+\eps \, .
$$
Taking $R>0$ large enough such that $\text{supp } v\subset B_R$ we obtain $\lambda_1(B_R)<\lambda_1(M)+\eps$.

Using the same argument introduced in the proof of Theorem \ref{main} it is not restrictive assuming that \eqref{equa} admits a positive solution $u\in H^1(M)$. For any $R>0$ denote by $\varphi_{1,R}\in H^1_0(B_R)$ a positive eigenfunction of $-\Delta_g$ corresponding to the first eigenvalue $\lambda_1(B_R)$. Then by $(f4)$ and integration by parts we obtain
\begin{align*}
0 & \le \int_{B_R} h(r,u)\varphi_{1,R} \, dV_g\\
& =\int_{B_R} (-\Delta_g \varphi_{1,R}-\lambda \varphi_{1,R})u \, dV_g+\psi^{n-1}(R)\int_{{\mathbb S}^{n-1}} u(R,\Theta) \frac{\partial \varphi_{1,R}}{\partial r}(R,\Theta)\, d\Theta\\
& \le [\lambda_1(B_R)-\lambda] \int_{B_R} u\varphi_{1,R} \, dV_g
\end{align*}
which yields $\lambda_1(B_R)\ge \lambda$ being $\int_{B_R} u\varphi_{1,R} \, dV_g>0$ . Passing to the limit as $R\to +\infty$ we obtain $\lambda_1(M)\ge \lambda$, a contradiction. \hfill $\Box$\par\vskip3mm

\section{Proof of Proposition \ref{p:existence}} \label{s:existence-2}

Since $V_\infty<\lambda_1(M)$ the Hilbert space $H^1(M)$ may be endowed with the following equivalent norm
$$
\|u\|_{*}:= \left(\int_{M} |\nabla_g u |_{g}^2\,dV_{g}-V_{\infty} \int_{M} u ^2\,dV_{g}\right)^{1/2} \, .
$$
We also introduce the functional
$$
\Phi_{\infty} (u)= \frac{1}{2} \|u\|_{*}^2-\int_M H(u)\,dV_{g}
$$
and then corresponding Nehari manifold
$$
\mathcal N_\infty:=\{v\in H^1(M)\setminus\{0\}:\langle \Phi'_\infty(u),u\rangle=0\} \, .
$$
We also put $c_\infty:=\ds{\inf_{v\in \mathcal N_\infty} \Phi_\infty(v)}$.

We observe that the assumptions \eqref{eq:bounds-V}-\eqref{eq:AR-h} are completely similar to $(f2)-(f4)$ of Theorem \ref{compact 1} so that one may follow closely the first part of the proof of Theorem \ref{compact 1} and prove the existence of a bounded Palais-Smale sequence $\{u_k\}\subset \mathcal N$ for $\Phi$ such that  $\Phi(u_k)\to c$. Up to subsequences we may suppose that $\{u_k\}$ is weakly convergent in $H^1(M)$.

We divide the proof into two cases.

{\bf Case 1.} Suppose that $u_k\rightharpoonup u$ with $u\not\equiv 0$. Then in a standard way one can show that $u$ is a nontrivial solution of \eqref{eq:V(x)}. This means that $u\in \mathcal N$. It remains to show that $u$ is a least energy solution of \eqref{eq:V(x)}. Since $\{u_k\}$ is a bounded Palais-Smale sequence, by \eqref{eq:AR-h} and the Fatou Lemma we have
\begin{align*}
c&=\lim_{k\to +\infty} \left(\Phi(u_k)+\frac 12 \langle \Phi'(u_k),u_k\rangle\right)=\lim_{k\to +\infty} \frac 12 \int_M [h(u_k)u_k-2H(u_k)] \, dV_g \\
& \ge \frac 12 \int_M [h(u)u-2H(u)] \, dV_g=\Phi(u)+\frac 12 \langle \Phi'(u),u \rangle=\Phi(u) \, .
\end{align*}
This proves that $u$ is a least energy solution.

{\bf Case 2.} Suppose that $u_k \rightharpoonup 0$. By the continuous embedding $H^1(M)\subset L^{2^*}(M)$ we also have $u_k \rightharpoonup 0$ in $L^{p+1}(M)$ with $p$ as in \eqref{eq:subcrit}. We claim that $u_k\not\rightarrow 0$ in $L^{p+1}(M)$. By \eqref{eq:subcrit} we deduce that for any $\eps>0$ there exists $C_\eps$ such that
\begin{equation} \label{eq:h-eps-Ceps}
|h(s)s|\le \eps s^2+C_\eps |s|^{p+1} \, .
\end{equation}
If we suppose by contradiction that $u_k\to 0$ in $L^{p+1}(M)$ then by \eqref{eq:h-eps-Ceps} with $\eps<\lambda_1(M)-\sup V$ we obtain
\begin{align*}
\limsup_{k\to +\infty} & \left(\int_M |\nabla_g u_k|_g^2 \, dV_g-(\sup V+\eps)\int_M u_k^2 \, dV_g\right)\\
& \le \limsup_{k\to +\infty} \left(\langle \Phi'(u_k),u_k\rangle + C_\eps \int_M |u_k|^{p+1} \, dV_g\right)=0
\end{align*}
thus proving that $u_k\to 0$ strongly in $H^1(M)$. This is absurd since $\mathcal N$ is bounded away from zero as one can show proceeding as in the proof of Lemma \ref{l:nehari-1} (iv). In the geometric assumptions of this proposition, by \cite[Lemma 2.26]{au} there exists a uniformly locally finite covering of $M$ by balls of a fixed radius. Then we can apply  \cite[Lemma 2.1.2]{marzuola} to deduce that there exist $\delta>0$ and $\{y_k\}\subset M$ such that
\begin{equation*}
\int_{B(y_k,1)} u_k^2 \, dV_g\ge \delta
\end{equation*}
where $B(y_k,1)$ denotes the geodesic ball of radius $1$ centered at $y_k$. The sequence $\{y_k\}$ is necessarily unbounded since otherwise we would have
$u_k \rightharpoonup u\neq 0$. Hence it is not restrictive assuming, up to subsequences, that $d(y_k,o)\to +\infty$ where $o$ is a point of $M$ fixed arbitrarily.
Since $M$ is weakly homogeneous there exists a sequence of isometries $\{\gamma_k\}$ such that $d(\gamma_k(y_k),o)\le D$ for some $D>0$. Then, for any $k\in \N$, we define $\widehat u_k(x):=u_k(\gamma_k^{-1}(x))$ so that $\{\widehat u_k\}$ is bounded in $H^1(M)$ and $\int_{B(o,D+1)} \widehat u_k^2 \, dV_g\ge \delta$.
This shows that $\widehat u_k\rightharpoonup \widehat u\neq 0$ in $H^1(M)$ up to subsequences.

We first complete the proof of the theorem when $V$ is a constant function. In this way $\{\widehat u_k\}$ is still a Palais-Smale sequence for $\Phi$ such that $\Phi(\widehat u_k)\to c$ and its weak limit $\widehat u$ is nontrivial. We found a nontrivial critical point of $\Phi$ which is also a least energy solution as one can prove by proceeding as in Case 1.

It remains to consider the general case when $V$ is not necessarily constant. By \eqref{eq:V(infty)} we deduce that $\widehat u$ is a nontrivial critical point for the functional $\Phi_\infty$ and in particular $\widehat u\in \mathcal N_\infty$ and $\Phi_\infty(\widehat u)\ge c_\infty$. Therefore by \eqref{eq:AR-h} and Fatou Lemma
\begin{align} \label{eq:c<=c-infty}
c&=\lim_{k\to +\infty} \left(\Phi(u_k)-\frac{1}{2}\langle \Phi'(u_k),u_k\rangle\right)=\lim_{k\to +\infty} \frac 12  \int_M\left[h(u_k)u_k-2H(u_k)\right] \,dV_g \\
\notag & = \lim_{k\to +\infty} \frac 12 \int_M\left[h(\widehat{u}_k)\widehat{u}_k-2H(\widehat{u}_k)\right]\,dV_g\geq \frac 12 \int_M\left[h(\widehat u)\widehat u-2H(\widehat u)\right]\,dV_g\\
\notag &=\Phi_{\infty}(\widehat u)\geq c_{\infty} \, .
\end{align}
On the other hand, by \eqref{eq:bounds-V} and \eqref{eq:AR-h} we have that $c$ and $c_\infty$ admits the following minimax characterization
\begin{equation} \label{eq:minimax}
c=\inf_{v\in H^1(M)\setminus \{0\}} \max_{t>0} \Phi(tv) \quad \text{and} \quad c_\infty=\inf_{v\in H^1(M)\setminus \{0\}} \max_{t>0} \Phi_\infty(tv)
\end{equation}
see Lemma \ref{l:nehari-1} (iii) for more details. But $\Phi(v)\le \Phi_\infty(v)$ for any $v\in H^1(M)$ and hence by \eqref{eq:minimax} we obtain $c\le c_\infty$ which combined with \eqref{eq:c<=c-infty} gives $c=c_\infty$.

As shown above in the case $V$ constant, the functional $\Phi_\infty$ admits a least energy solution $w\in \mathcal N_\infty$. Proceeding as in Lemma \ref{positivity} on can show that $w$ does not change sign and moreover it is either strictly positive or strictly negative as a consequence of \eqref{eq:AR-h} and the strong maximum principle for variational solutions, see \cite[Section 8.7]{Gilb}. Therefore, if $V$ is not constant, i.e. $V>V_\infty$ on set of positive measure, we have
\begin{align*}
\Phi(w)&=\int_{M} |\nabla_g w|_g^2 \, dV_g-\int_M V(x)w^2 \, dV_g-\int_M H(w)\, dV_g\\
& <\int_{M} |\nabla_g w|_g^2 \, dV_g-V_\infty \int_M w^2 \, dV_g-\int_M H(w)\, dV_g=\Phi_\infty (w)=c_\infty  \, .
\end{align*}
This implies $c\le \Phi(w)<c_\infty$ a contradiction. \hfill $\Box$\par\vskip3mm

\bigskip

\section{Proof of Proposition \ref{p:constant-K}} \label{s:K-constant}

In the sequel we use the following notations:
\begin{align*}
B(x,R):=\{y\in M: d(y,x)<R\} \, , \qquad S(x,R):=\{y\in M:
d(y,x)=R\} \, .
\end{align*}
where $d$ is the geodesic distance.

We first prove that the scalar curvature $K$ is necessarily constant in $M$. Let $D>0$ and $\Gamma$ be as in Definition \ref{d:weakly-hom}.
Let $\rho>D$ be fixed arbitrarily and let $y\in
M$ be such that $d(o,y)=\rho$. Then there exists $\gamma\in \Gamma$ such
that $d(o,\gamma(y))\le D$. Moreover $\gamma(o)\neq o$ and $d(\gamma(o),o)\ge
\rho-D$. Indeed if we assume by contradiction that
$d(\gamma(o),o)<\rho-D$ then we would have
\begin{align*}
\rho=d(y,o)=d(\gamma(y),\gamma(o))\le d(\gamma(y),o)+d(o,\gamma(o))<D+\rho-D=\rho
\end{align*}
and this is absurd. Let $x:=\gamma(o)\neq o$ and denote by $\rho_x$ the
distance of $x$ from the pole. Since $K$ is constant on $S(o,R)$
for any $R>0$ and since $\gamma$ is an isometry then $K$ is constant on
$S(x,R)$ for any $R>0$. We denote by $K_R$ this constant value.

Let $0<R<\rho_x$ be fixed arbitrarily and consider the ball
$B(x,R)$. Let $z\in B(x,R)$, $\rho_z:=d(o,z)$.
We observe that
\begin{align*}
& B(x,R)\cap \{tx:t\in \R \}=\left\{\frac{t}{\rho_x}\, x:t\in (\rho_x-R,\rho_x+R)\right\}
\end{align*}
and hence $-\frac{\rho_z}{\rho_x} x\not\in B(x,R)$; but $-\frac{\rho_z}{\rho_x} x \in S(o,\rho_z)$ so that $S(o,\rho_z)\not\subseteq B(x,R)$ and in particular
$S(o,\rho_z)\cap S(x,R)\neq \emptyset$.

Since $K$ is constant on $S(x,R)$ and on $S(o,\rho_z)$ and $z\in S(o,\rho_z)$ then
$K(z)=K_R$. We have proved that $K(z)=K_R$ for any $z\in B(x,R)$
and for any $0<R<\rho_x$. Moreover we also have $K(z)=K_R$ for any
$z\in B(o,R)$ and for any $0<R<\rho_x$ thanks to the rotational symmetry of $(M,g)$ with respect to $o$. In particular $K_R=K(o)$ for any $R\in (0,\rho_x)$ and hence $K(z)=K(o)$ for any $z\in B(o,\rho_x)$. But we recall that $\rho_x\ge \rho-D$ and hence one may find $x$ such that $\rho_x$ is arbitrarily large simply choosing
$\rho$ large enough. In this way one proves that $K$ is a constant $\kappa$ over all $M$.

By $(H)$ and \eqref{eq:scalar} we then have
\begin{equation} \label{eq:unique}
\begin{cases}
\ds{2\psi''(r)+(n-2)\frac{(\psi'(r))^2-1}{\psi}+\beta\psi(r)=0 }\\
\psi(0)=0 \, ,\quad \psi'(0)=1\, , \quad \psi''(0)=0 \, .
\end{cases}
\end{equation}
where we put $\beta=\kappa/(n-1)$. We prove that \eqref{eq:unique} admits a unique solution.

Let $\psi_1,\psi_2$ be two solutions of \eqref{eq:unique}. Define $\phi_i(r):=\psi_i(r)/r$ for $i=1,2$ so that
\begin{equation}  \label{eq:unique-2}
\begin{cases}
\ds{\phi_i''(r)+\frac{n}r \phi_i'(r)+\frac{n-2}{2r^2}\frac{\phi_i^2(r)-1}{\phi_i(r)}+\frac{n-2}2 \frac{(\phi_i'(r))^2}{\phi_i(r)}+\frac{\beta}2\phi_i(r)=0} \\
\phi_i(0)=1\, , \quad \phi'_i(0)=0 \, .
\end{cases}
\end{equation}
Then we define $w(r)=\phi_1(r)-\phi_2(r)$ and we obtain
\begin{equation}  \label{eq:unique-3}
\begin{cases}
w''+\frac{n}r w'=-\frac{n-2}{2r^2}\frac{\phi_1(r)\phi_2(r)+1}{\phi_1(r)\phi_2(r)}w-\frac{n-2}2\left[ \frac{(\phi_1'(r))^2}{\phi_1(r)}-\frac{(\phi_2'(r))^2}{\phi_2(r)}\right]-\frac{\beta}2w \\
w(0)=0 \, , \quad w'(0)=0 \, .
\end{cases}
\end{equation}
Multiplying by $r^n$ and integrating we obtain
\begin{align*}
w'(r)=r^{-n}\int_0^r\left\{-\tfrac{n-2}{2}\tfrac{\phi_1\phi_2+1}{\phi_1\phi_2}t^{n-2}w-\tfrac{n-2}2\left[ \tfrac{(\phi_1')^2}{\phi_1}-\tfrac{(\phi_2')^2}{\phi_2}\right]t^n-\tfrac{\beta}2t^nw \right\}dt \, .
\end{align*}
Since $\lim_{t\to 0^+} \frac{\phi_1(t)\phi_2(t)+1}{\phi_1(t)\phi_2(t)}=2$ and since the function of two real variables $(z_1,z_2)\mapsto z_1^2/z_2$ is lipschitzian in a neighborhood of $(0,1)$, we deduce that for any $\eps>0$ there exist $C>0$ and $\delta>0$ such that for any $r\in (0,\delta)$ we have
\begin{align*}
|w'(r)|& \leq r^{-n}(n-2+\eps) \int_0^r t^{n-2} |w(t)|dt+C r^{-n} \int_0^r t^n \Big[|w(t)|+|w'(t)|\Big]dt\\
 & \le \left(\frac{n-2+\eps}{n-1} \frac 1r+\frac{C}{n+1} r\right) \sup_{t\in [0,r]} |w(t)|+C \int_0^r |w'(t)|dt\,.
\end{align*}
Since $w(t)=\int_0^t w'(s)ds$ we obtain
\begin{align*}
|w'(r)|\le \left(\frac{n-2+\eps}{n-1} \frac 1r+\frac{C}{n+1} r+C\right)  \int_0^r |w'(t)|dt \qquad \text{for any } r\in (0,\delta) \, .
\end{align*}
Up to shrinking $\delta$ if necessary we may assume that $\frac{n-2+\eps}{n-1} \frac 1r+\frac{C}{n+1} r+C<\frac{n-2+2\eps}{n-1}\frac 1r$ for any $r\in (0,\delta)$ in order to obtain
\begin{align} \label{eq:unique-4}
|w'(r)|\le \frac{n-2+2\eps}{n-1} \frac 1r  \int_0^r |w'(t)|dt \qquad \text{for any } r\in (0,\delta) \, .
\end{align}
If we put $h(r)=\int_0^r |w'(t)|dt$ then by \eqref{eq:unique-4} it follows
\begin{align*} 
h'(r)\le \frac{n-2+2\eps}{n-1} \frac 1r  h(r) \qquad \text{for any } r\in (0,\delta) \, .
\end{align*}
This implies that the map $r\mapsto r^{-\frac{n-2+2\eps}{n-1}} h(r)$ is nonincreasing in $(0,\delta)$ and hence if we choose $\eps=\frac 12$, for any $r\in (0,\delta)$ we have
\begin{equation*}
0\le r^{-1} h(r)\le \lim_{t\to 0^+} t^{-1} h(t)= \lim_{t\to 0^+} \tfrac{\int_0^t |w'(s)|ds}{t}=\lim_{t\to 0^+} |w'(t)|=0 \, .
\end{equation*}
This proves that $h\equiv 0$ in $(0,\delta)$ and hence, in turn, we also have $w'\equiv 0$; since $w(0)=0$ then $w\equiv 0$ in $(0,\delta)$ and
by the definition of $w, \phi_1,\phi_2$ we finally obtain $\psi_1\equiv \psi_2$ in $(0,\delta)$.
The fact that $\psi_1$ and $\psi_2$ coincide over all $(0,+\infty)$ follows immediately from classical uniqueness for Cauchy problems.

We claim that $\kappa\le 0$, otherwise if $\kappa>0$, by uniqueness, we would have $\psi(r)=\alpha^{-1} \sin(\alpha r)$ where we put
$\alpha=\sqrt{\frac{\kappa}{n(n-1)}}$, in contradiction with $(H)$. The two alternatives in Proposition \ref{p:constant-K} follow as well. \hfill $\Box$\par\vskip3mm

\bigskip\noindent
\textbf{Acknowledgements.} The authors are grateful to Antonio J. Di Scala and Gabriele Grillo for fruitful
discussions during the preparation of this paper.

\noindent
This work has been partially supported by the Research Project FIR (Futuro in Ricerca) 2013 \emph{Geometrical and qualitative aspects of PDE's}.

 \noindent
The authors are members of the Gruppo Nazionale per l'Analisi Matematica, la Probabilit\`a e le loro Applicazioni (GNAMPA) of the Istituto Nazionale di Alta Matematica (INdAM).

\bigskip

\bibliographystyle{amsplain} 
\end{document}